\documentclass{amsart}

\usepackage{amssymb,amsfonts}
\usepackage[neverdecrease,pointedenum]{paralist}
\usepackage{multirow}
\usepackage{url}

\newcommand{\card}[1]{\ensuremath{\lvert{#1}\rvert}}     % cardinality
\newcommand{\cl}[1]{\ensuremath{\mathcal{#1}}}           % function class

\newcommand{\nset}[1]{\ensuremath{[{#1}]}}               % n-element set
\newcommand{\couples}[1][n]{\ensuremath{\binom{#1}{2}}}  % the set of 2-element subsets of an n-element set
\newcommand{\ms}[1]{\langle #1 \rangle}     % multiset

\DeclareMathOperator{\deck}{deck}                        % deck

\theoremstyle{plain}
\newtheorem{theorem}{Theorem}[section]

\newtheorem{lemma}[theorem]{Lemma}

\theoremstyle{definition}

\newtheorem{example}[theorem]{Example}

\theoremstyle{remark}
\newtheorem{remark}[theorem]{Remark}

\begin{document}
\noindent
{\footnotesize
Preprint of an article published in
\textit{International Journal of Algebra and Computation}
(2014),
DOI: 10.1142/S0218196714500027.
\copyright{} copyright World Scientific Publishing Company.
\newline
\url{http://www.worldscientific.com/worldscinet/ijac}
}

\bigskip
\bigskip

\title[Reconstructing multisets over commutative groupoids]{Reconstructing multisets over commutative groupoids and affine functions over nonassociative semirings}

\author{Erkko Lehtonen}

\address[E. Lehtonen]{Computer Science and Communications Research Unit \\
University of Luxembourg \\
6, rue Richard Coudenhove-Kalergi, L--1359 Luxembourg, Luxembourg}
\email{erkko.lehtonen@uni.lu}

\begin{abstract}
A reconstruction problem is formulated for multisets over commutative groupoids. The cards of a multiset are obtained by replacing a pair of its elements by their sum. Necessary and sufficient conditions for the reconstructibility of multisets are determined. These results find an application in a different kind of reconstruction problem for functions of several arguments and identification minors: classes of linear or affine functions over nonassociative semirings are shown to be weakly reconstructible. Moreover, affine functions of sufficiently large arity over finite fields are reconstructible.
\end{abstract}

\maketitle

%%%%%%%%%%%%%%%%%%%%%%%%%%%%%%%%%%%%%%%%%%%%%%%%%%

\section{Introduction}

Generally speaking, a reconstruction problem asks whether a mathematical object can be recovered from partial information. The kind of reconstruction problems we discuss in this paper fall into the following general scheme: given a combinatorial object, we apply a certain operation to its elements in all possible ways, and we ask whether the initially given object is uniquely determined (up to some kind of isomorphism) by the collection of these derived objects (which are called the cards of the object). Perhaps one of the most famous reconstruction problems is the following: Is every graph uniquely determined, up to isomorphism, by the collection of its subgraphs obtained by deleting a single vertex? It was conjectured by Kelly~\cite{Kelly1942} (see also Ulam's problem book~\cite{Ulam}) that the answer is positive, provided that the graph has at least three vertices. While the conjecture has been shown to hold for various classes of graphs, in full generality it remains an important open problem in graph theory.

In this paper we consider two different reconstruction problems. The first one deals with multisets over a commutative groupoid. The cards of a multiset $M$ are the multisets obtained from $M$ by replacing a pair of its elements by their sum. We find necessary and sufficient conditions for a multiset to be reconstructible in this setting. Reconstructibility depends on the cardinality and the form of a multiset and also on the underlying groupoid.

The reconstruction problem for multisets arose from a completely different reconstruction problem that was posed and studied in~\cite{LehtonenDeckSymm}. Here, the objects are functions of several arguments, and the cards of a function $f \colon A^n \to B$ are the $\couples$ identification minors of $f$, i.e., the $(n-1)$-ary functions obtained from $f$ by identifying a pair of its arguments. The special case of affine functions over semirings reduces to the reconstruction problem for multisets over commutative groupoids. As an application of our solution to the reconstruction problem for multisets, we show that classes of affine or linear functions over nonassociative semirings are weakly reconstructible. Moreover, affine functions of sufficiently large arity over finite fields are reconstructible.

A function $f \colon A^n \to B$ is called a minor of another function $g \colon A^m \to B$ if there exists a map $\sigma \colon \{1, \dots, m\} \to \{1, \dots, n\}$ such that $f(a_1, \dots, a_n) = g(a_{\sigma(1), \dots, \sigma(m)})$ for all $a_1, \dots, a_n \in A$.
New discoveries -- either positive or negative -- about the reconstruction problem for functions and identification minors will shed some light on the minor relation, a topic that has attracted the attention of several researchers over the past years (see, e.g., \cite{BCP,CouFol2005,CouLeh2009,EFHH,LehSze,Pippenger,Wang,Willard,Zverovich}).

\section{Preliminaries}
\label{sec:preliminaries}

We follow the standard terminology and notation of abstract algebra, as found, e.g., in~\cite{Lang,MacBir}.

For a positive integer $n$, the set $\{1, \dots, n\}$ is denoted by $\nset{n}$. The set of all $2$-element subsets of $\nset{n}$ is denoted by $\couples$. The set $\{0, 1, 2, \dots\}$ of nonnegative integers is denoted by $\mathbb{N}$.

A \emph{finite multiset} $M$ over a set $X$ is a map $\mathbf{1}_M \colon X \to \mathbb{N}$, called a \emph{multiplicity function,} such that the set $\{x \in X : \mathbf{1}_M(x) \neq 0\}$ is finite. We will only discuss finite multisets, and we will refer to them simply as \emph{multisets.} For a finite multiset $M$, the sum $\sum_{x \in X} \mathbf{1}_M(x)$ is a well-defined natural number, and it is called the \emph{cardinality} of $M$ and denoted by $\card{M}$. We refer to a multiset of cardinality $n$ as an \emph{$n$-multiset.} For each $x \in X$, the number $\mathbf{1}_M(x)$ is called the \emph{multiplicity} of $x$ in $M$.
We write $x \in M$ if $\mathbf{1}_M(x) \geq 1$.
A multiset $M$ is a \emph{submultiset} of $M'$, denoted $M \subseteq M'$, if $\mathbf{1}_M(x) \leq \mathbf{1}_{M'}(x)$ for all $x \in X$.
The \emph{empty multiset} $\emptyset$ is given by the multiplicity function $\mathbf{1}_\emptyset(x) = 0$ for all $x \in X$.

We may represent a finite multiset $M$ as a list enclosed in angle brackets where each element $x \in X$ occurs $\mathbf{1}_M(x)$ times, e.g., $\ms{0, 0, 0, 1, 1, 2}$. Also, if $(a_i)_{i \in I}$ is a finite indexed family of elements of $X$, then we will write $\ms{a_i : i \in I}$ to denote the multiset in which the multiplicity of each $x \in X$ equals $\card{\{i \in I : a_i = x\}}$.

Let $M$ and $M'$ be finite multisets over $X$. The \emph{multiset sum} $M \uplus M'$, the \emph{difference} $M \setminus M'$, and the \emph{intersection} $M \cap M'$ of $M$ and $M'$ are defined by the multiplicity functions
\begin{align*}
\mathbf{1}_{M \uplus M'}(x) &= \mathbf{1}_M(x) + \mathbf{1}_{M'}(x), \\
\mathbf{1}_{M \setminus M'}(x) &= \max(\mathbf{1}_M(x) - \mathbf{1}_{M'}(x), 0),\\
\mathbf{1}_{M \cap M'}(x) &= \min(\mathbf{1}_M(x), \mathbf{1}_{M'}(x)).
\end{align*}

%%%%%%%%%%%%%%%%%%%%%%%%%%%%%%%%%%%%%%%%%%%%%%%%%%

\section{Reconstruction problem for multisets over commutative groupoids}

In this section, we formulate a reconstruction problem for multisets over commutative groupoids, and we completely characterize the reconstructible multisets. We will conclude the section with some open problems.

\subsection{Reconstruction problem for multisets}

Let $(G; +)$ be a commutative groupoid. Let $n$ be an integer at least $2$. Let $M$ be a multiset of cardinality $n$ over $G$. Fix an $n$-tuple $(m_1, \dots, m_n) \in G^n$ satisfying $M = \ms{m_1, \dots, m_n}$. For each $I \in \couples$, let $M_I := M \setminus \ms{m_{\min I}, m_{\max I}} \uplus \ms{m_{\min I} + m_{\max I}}$.
The \emph{cards} of $M$ are the multisets $M_I$, for each $I \in \couples$, and the \emph{deck} of $M$ is the multiset $\deck M := \ms{M_I : I \in \couples}$.
(It is irrelevant which particular tuple $(m_1, \dots, m_n)$ is chosen for a given $M$. Any valid choice will give rise to the same deck.)

A multiset $M$ is \emph{reconstructible} if for all multisets $M'$ over $G$ the condition $\deck M = \deck M'$ implies $M = M'$.

We have now all the necessary definitions for our discussion, and we can formulate a reconstruction problem for multisets over commutative groupoids. Is every multiset over a commutative groupoid reconstructible? If not, which multisets are reconstructible and which ones are not? Does the answer depend on the underlying groupoid?

Examples~\ref{ex:4rstab}--\ref{ex:2abcd} below illustrate that the answer to our first question is negative: not every multiset over a commutative groupoid is reconstructible. The latter two questions will be answered later in this paper. It will turn out that these examples are exhaustive; there is no other nonreconstructible multiset over any commutative groupoid than the ones described in Examples~\ref{ex:4rstab}--\ref{ex:2abcd}.

\begin{example}
\label{ex:4rstab}
Let $(G; +)$ be a commutative groupoid with elements $r$, $s$, $t$, $u$, $v$ satisfying $x + u = v$ and $x + v = u$ for all $x \in \{r, s, t\}$ and $r + s = s$, $s + t = t$, $t + r = r$.
Let $M = \ms{r, s, t, u}$, $M' = \ms{r, s, t, v}$. If $u \neq v$, then $M \neq M'$ but
\[
\deck M = \deck M' =
\ms{ \ms{r, s, u}, \ms{r, t, u}, \ms{s, t, u}, \ms{r, s, v}, \ms{r, t, v}, \ms{s, t, v} }.
\]

Note that either $r$, $s$ and $t$ are pairwise distinct or $r = s = t$. For, assume that $r = s$. Then $t = s + t = r + t = r$, whence $r = s = t$. A similar argument shows that $r = t$ or $s = t$ implies $r = s = t$.

If $r$, $s$ and $t$ are pairwise distinct, then $r + (s + t) = r \neq t = (r + s) + t$, i.e., $(G; +)$ is not associative. If $r = s = t$ and $u \neq v$, then $r + (r + u) = u \neq v = (r + r) + u$, i.e., $(G; +)$ is not associative and not even alternative. (A binary operation is \emph{left alternative} if it satisfies the identity $x(xy) = (xx)y$ and \emph{right alternative} if it satisfies the identity $y(xx) = (yx)x$. An operation is \emph{alternative} if it is both left and right alternative. Alternativity is a weaker form of associativity.)
\end{example}

\begin{example}
\label{ex:3abaceqbc}
Let $(G; +)$ be a commutative groupoid with elements $r$, $s$, $t$ satisfying $r + (r + s) = s$, $r + (r + t) = t$ and $(r + s) + (r + t) = s + t$.
Let $M = \ms{r, s, t}$, $M' = \ms{r, r + s, r + t}$. Then
\[
\deck M = \deck M' =
\ms{ \ms{r, s + t}, \ms{s, r + t}, \ms{t, r + s} }.
\]
Furthermore, if $\{r + s, r + t\} \neq \{s, t\}$, then $M \neq M'$.

Note that if $(G; +)$ is a Boolean group (a group in which every nonneutral element has order $2$), then the above conditions are satisfied by all elements $r, s, t \in G$. Furthermore, if $r$ is not neutral and $r + s \neq t$, then $M \neq M'$.
\end{example}

\begin{example}
\label{ex:3abaceqa}
Let $(G; +)$ be a commutative groupoid with elements $r$, $s$, $t$ satisfying $(r + s) + (r + t) = r$, $(r + s) + (s + t) = s$ and $(r + t) + (s + t) = t$.
Let $M = \ms{r, s, t}$, $M' = \ms{r + s, r + t, s + t}$. Then
\[
\deck M = \deck M' =
\ms{ \ms{r, s + t}, \ms{s, r + t}, \ms{t, r + s} },
\]
but $M$ and $M'$ are not necessarily equal.

Note that if $(G; +)$ is a monoid with neutral element $0$, then the above conditions are satisfied by all elements $r, s, t \in G$ such that $r + s + t = 0$.
\end{example}

\begin{example}
\label{ex:2abcd}
Let $(G; +)$ be a commutative groupoid with elements $r$, $s$, $t$, $u$ satisfying $r + s = t + u$ and $\{r, s\} \neq \{t, u\}$. Let $M = \ms{r, s}$, $M' = \ms{t, u}$. Then $M \neq M'$ but
$\deck M = \deck M' = \ms{\ms{r + s}}$.
\end{example}

\subsection{The solution to the reconstruction problem for multisets}

We are going to show that a multiset over a commutative groupoid is reconstructible if and only if it is not any one of the multisets described in Examples~\ref{ex:4rstab}--\ref{ex:2abcd}. In our approach to characterizing the reconstructible multisets, we will make good use of the collection of all elements in all cards of a multiset, as it reveals plenty of information about the multiset itself.
Let $M$ be a multiset of cardinality $n$ over a commutative groupoid $(G; +)$.
Denote $\widetilde{M} := \biguplus_{I \in \couples} M_I$, and denote $N_M(x) := \mathbf{1}_{\widetilde{M}}(x)$ for each $x \in G$.

\begin{lemma}
\label{lem:deltaM}
Let $M$ be a multiset of cardinality $n$ over $G$. Then $N_M(x) = \mathbf{1}_M(x) \cdot \binom{n-1}{2} + \delta_M(x)$ for some $\delta_M \colon G \to \mathbb{N}$ satisfying $\sum_{x \in G} \delta_M(x) = \binom{n}{2}$.
\end{lemma}

\begin{proof}
Fix an $n$-tuple $(m_1, \dots, m_n)$ satisfying $M = \ms{m_1, \dots, m_n}$ and for each $I \in \couples$, let $M_I := M \setminus \ms{m_{\min I}, m_{\max I}} \uplus \ms{m_{\min I} + m_{\max I}}$. Let us count the number of times each element of $G$ occurs in the various cards of $M$. For each $i \in \nset{n}$, there is an occurrence of $m_i$ (that has not yet been counted) in $M_I$ for every $I \in \couples$ such that $i \notin I$. Additionally, for each $I \in \couples$, there is an occurrence of $m_{\min I} + m_{\max I}$ in $M_I$. In other words, each occurrence of $x$ in $M$ contributes $\binom{n-1}{2}$ to the number $\mathbf{1}_{\widetilde{M}}(x)$, and each $I \in \couples$ contributes $1$ to the number $\mathbf{1}_{\widetilde{M}}(m_{\min I} + m_{\max I})$.
Let $\delta_M \colon G \to \mathbb{N}$ be the map given by the rule $\delta_M(x) = \card{\{I \in \couples : m_{\min I} + m_{\max I} = x\}}$.
Then clearly $\sum_{x \in G} \delta_M(x) = \binom{n}{2}$ and we have $\mathbf{1}_{\widetilde{M}}(x) = \mathbf{1}_M(x) \cdot \binom{n-1}{2} + \delta_M(x)$ for all $x \in G$.
\end{proof}

\begin{lemma}
\label{lem:M'eqMab}
Assume that $n \geq 4$, and let $M$ and $M'$ be multisets of cardinality $n$ over $G$. Assume that $N_M(x) = N_{M'}(x)$ for all $x \in G$ and there exists $y \in G$ such that $N_M(y)$ is not a multiple of $\binom{n-1}{2}$. Then there exist elements $a, b \in G$ such that $M' = M \setminus \ms{a} \uplus \ms{b}$.
\end{lemma}

\begin{proof}
If $M = M'$, then the claim clearly holds with $a = b$ for any $a \in M$.
Assume that $M \neq M'$.
Then there exist distinct elements $a$ and $b$ of $G$ such that $\mathbf{1}_M(a) \neq \mathbf{1}_{M'}(a)$ and $\mathbf{1}_M(b) \neq \mathbf{1}_{M'}(b)$.

Observe that $\binom{n-1}{2} < \binom{n}{2} \leq 2 \cdot \binom{n-1}{2}$ whenever $n \geq 4$ and the second inequality holds with equality if and only if $n = 4$.
By Lemma~\ref{lem:deltaM}, there exist maps $\delta_M, \delta_{M'} \colon G \to \mathbb{N}$ such that $N_M(x) = \mathbf{1}_M(x) \cdot \binom{n-1}{2} + \delta_M(x)$ and $N_{M'}(x) = \mathbf{1}_{M'}(x) \cdot \binom{n-1}{2} + \delta_{M'}(x)$ for all $x \in G$ and $\sum_{x \in G} \delta_M(x) = \binom{n}{2} = \sum_{x \in G} \delta_{M'}(x)$.
The assumption that $N_M(y)$ is not a multiple of $\binom{n-1}{2}$ implies that $\delta_M(x) \geq \binom{n-1}{2}$ for at most one $x \in G$ and $\delta_M(x) < 2 \cdot \binom{n-1}{2}$ for all $x \in G$; similarly $\delta_{M'}(x) \geq \binom{n-1}{2}$ for at most one $x \in G$ and $\delta_{M'}(x) < 2 \cdot \binom{n-1}{2}$ for all $x \in G$. We may assume, without loss of generality, that $\delta_M(a) < \binom{n-1}{2}$. This implies that $\mathbf{1}_{M'}(a) = \mathbf{1}_M(a) - 1$ and $\delta_{M'}(a) = \delta_M(a) + \binom{n-1}{2} \geq \binom{n-1}{2}$. This in turn implies that $\delta_{M'}(x) < \binom{n-1}{2}$ for all $x \in G \setminus \{a\}$; in particular, $\delta_{M'}(b) < \binom{n-1}{2}$. Consequently, $\mathbf{1}_M(b) = \mathbf{1}_{M'}(b) - 1$ and $\delta_M(b) = \delta_{M'}(b) + \binom{n-1}{2} \geq \binom{n-1}{2}$. It also follows that $\mathbf{1}_M(x) = \mathbf{1}_{M'}(x)$ for all $x \in G \setminus \{a, b\}$ (for, if there existed an element $c \in G \setminus \{a, b\}$ such that $\mathbf{1}_M(c) \neq \mathbf{1}_{M'}(c)$, then this would imply, similarly as above, that $\delta_M(c) \geq \binom{n-1}{2}$, which would contradict the fact that there is at most one $x \in G$ such that $\delta_M(x) \geq \binom{n-1}{2}$). We conclude that $\mathbf{1}_{M'}(a) = \mathbf{1}_M(a) - 1$, $\mathbf{1}_{M'}(b) = \mathbf{1}_M(b) + 1$ and $\mathbf{1}_{M'}(x) = \mathbf{1}_M(x)$ for all $x \in G \setminus \{a, b\}$. In other words, $M' = M \setminus \ms{a} \uplus \ms{b}$.
\end{proof}

\begin{theorem}
\label{thm:ngeq5}
Assume that $(G; +)$ is a commutative groupoid, and $M$ and $M'$ are multisets over $G$ with $\card{M} = \card{M'} \geq 5$. Then $\deck M = \deck M'$ if and only if $M = M'$.
\end{theorem}

\begin{proof}
Let $n$ be the common cardinality of $M$ and $M'$. It is clear that if $M = M'$, then $\deck M = \deck M'$. For the converse implication, assume that $\deck M = \deck M'$. Then clearly $\widetilde{M} = \widetilde{M}'$, so $N_M(x) = N_{M'}(x)$ for all $x \in G$. Since $\sum_{x \in G} \delta_M(x) = \binom{n}{2}$ and $\binom{n-1}{2} < \binom{n}{2} < 2 \cdot \binom{n-1}{2}$ holds whenever $n \geq 5$, there exists $y \in G$ such that $N_M(y)$ is not a multiple of $\binom{n-1}{2}$. By Lemma~\ref{lem:M'eqMab}, there exist $a, b \in G$ such that $M' = M \setminus \ms{a} \uplus \ms{b}$, say, $M = \ms{m_1, \dots, m_{n-1}, a}$, $M' = \ms{m_1, \dots, m_{n-1}, b}$ for some $m_1, \dots, m_{n-1} \in G$.

Let us count the number of times each element of $G$ occurs in the multisets $\widetilde{M}$ and $\widetilde{M}'$. Both multisets contain $\binom{n-1}{2}$ occurrences of $m_i$ for each $i \in \nset{n-1}$ and one occurrence of $m_{\min I} + m_{\max I}$ for each $I \in \couples$ such that $n \notin I$. The remaining elements of $\widetilde{M}$ are $\binom{n-1}{2}$ occurrences of $a$ and one occurrence of $m_i + a$ for each $i \in \nset{n-1}$; while the remaining elements of $\widetilde{M}'$ are $\binom{n-1}{2}$ occurrences of $b$ and one occurrence of $m_i + b$ for each $i \in \nset{n-1}$. Since $\binom{n-1}{2} > n - 1$ whenever $n \geq 5$, the equality $\widetilde{M} = \widetilde{M}'$ may hold only if $a = b$. We conclude that $M = M'$.
\end{proof}

\begin{theorem}
\label{thm:neq4}
Assume that $(G; +)$ is a commutative groupoid. Let $M$ and $M'$ be multisets of cardinality $4$ over $G$. Then $\deck M = \deck M'$ if and only if one of the following conditions holds:
\begin{enumerate}[\rm (i)]
\item $M = M'$.
\item $M = \ms{r, s, t, u}$ and $M' = \ms{r, s, t, v}$ for some elements $r, s, t, u, v \in G$ satisfying $x + u = v$ and $x + v = u$ for all $x \in \{r, s, t\}$ and $r + s = s$, $s + t = t$, $t + r = r$.
\end{enumerate}
\end{theorem}

\begin{proof}
Let $n = 4$. Then $\binom{n-1}{2} = 3 = n - 1$.
It is clear that if $M = M'$, then $\deck M = \deck M'$. If condition (ii) holds, then $\deck M = \deck M'$, as shown in Example~\ref{ex:4rstab}. For the converse implication, assume that $\deck M = \deck M'$. Then obviously $\widetilde{M} = \widetilde{M}'$ and $N_M(x) = N_{M'}(x)$ for all $x \in G$.

Assume first that there is $y \in G$ such that $N_M(y)$ is not a multiple of $\binom{n-1}{2}$. By Lemma~\ref{lem:M'eqMab}, there exist elements $u, v \in G$ such that $M' = M \setminus \ms{u} \uplus \ms{v}$. If $u = v$, then $M = M'$ and we are done. Assume thus that $u \neq v$. Then $M = \ms{r, s, t, u}$ and $M' = \ms{r, s, t, v}$ for some $r, s, t \in G$, and the cards of $M$ and $M'$ are
\begin{align*}
M_{12} &= \ms{r + s, t, u}, & M'_{12} &= \ms{r + s, t, v}, \\
M_{13} &= \ms{r + t, s, u}, & M'_{13} &= \ms{r + t, s, v}, \\
M_{23} &= \ms{s + t, r, u}, & M'_{23} &= \ms{s + t, r, v}, \\
M_{14} &= \ms{r + u, s, t}, & M'_{14} &= \ms{r + v, s, t}, \\
M_{24} &= \ms{s + u, r, t}, & M'_{24} &= \ms{s + v, r, t}, \\
M_{34} &= \ms{t + u, r, s}, & M'_{34} &= \ms{t + v, r, s}.
\end{align*}

We must have $r + u = s + u = t + u = v$ and $r + v = s + v = t + v = u$. (Otherwise we would have $\widetilde{M} \neq \widetilde{M}'$, a contradiction.) Furthermore, $r$, $s$ and $t$ are not all equal. (For, if $r = s = t$, then $r + s = r + t = s + t$ and $N_M(x)$ would be a multiple of $\binom{n-1}{2}$ for all $x \in G$, a contradiction.)

Since $\deck M = \deck M'$, there exists a one-to-one correspondence between the cards of $M$ and the cards of $M'$. We are going to determine the possible correspondences. Observe first that $M_I \neq M'_I$ for all $I \in \couples$. Let us now focus on the card $M_{14}$ of $M$.

Suppose first that $M_{14} = M'_{24}$, i.e., $\ms{v, s, t} = \ms{u, r, t}$. This implies that $r = v$ and $s = u$.
Then $M_{13} = \ms{u, u, u}$, and this may correspond only to $M'_{14} = \ms{u, u, t}$ (because the other $M'_{ij}$ contain $v$). Thus, $t = u$. But then every card of $M$ contains $u$, so no card of $M$ can be equal to $M'_{23} = \ms{v, v, v}$, and we have reached a contradiction. We conclude that $M_{14} \neq M'_{24}$; in a similar way, we can deduce also that $M_{24} \neq M'_{34}$ and $M_{34} \neq M'_{14}$.

Suppose then that $M_{14} = M'_{34}$, i.e., $\ms{v, s, t} = \ms{u, r, s}$.
This implies that $r = v$ and $t = u$. A similar argument as above (now $M_{12} = \ms{u, u, u}$, $M'_{23} = \ms{v, v, v}$) leads to a contradiction. We conclude that $M_{14} \neq M'_{34}$; similarly, $M_{24} \neq M'_{14}$ and $M_{34} \neq M'_{24}$.

Suppose then that $M_{14} = M'_{23}$, i.e., $\ms{v, s, t} = \ms{s + t, r, v}$.
This implies that $\{s, t\} = \{s + t, r\}$, and we must have $\{M_{24}, M_{34}\} = \{M'_{12}, M'_{13}\}$.
If $M_{24} = M'_{12}$ and $M_{34} = M'_{13}$, then $\{r, t\} = \{r + s, t\}$ and $\{r, s\} = \{r + t, s\}$.
Consequently, $r = r + s$ and $r = r + t$.
From the equality $\{s, t\} = \{s + t, r\}$ we get that $s = s + t$ and $t = r$; or $s = r$ and $t = s + t$.
In either case, it follows that $r = s = t$, a contradiction.
If $M_{24} = M'_{13}$ and $M_{34} = M'_{12}$, then $\{r, t\} = \{r + t, s\}$ and $\{r, s\} = \{r + s, t\}$.
By these equalities and by the equality $\{s, t\} = \{s + t, r\}$, we have $r = s$ or $r = t$; and $s = r$ or $s = t$; and $t = r$ or $t = s$. It follows that $r = s = t$, again a contradiction.
We conclude that $M_{14} \neq M'_{23}$; similarly, $M_{24} \neq M'_{13}$ and $M_{34} \neq M'_{12}$.

Consider then the case that $M_{14} = M'_{12}$, i.e., $\ms{v, s, t} = \ms{r + s, t, v}$.
Then we must have $M_{24} = M'_{23}$ and $M_{34} = M'_{13}$, i.e., $\ms{v, r, t} = \ms{s + t, r, v}$ and $\ms{v, r, s} = \ms{r + t, s, v}$.
It follows that $r + s = s$, $s + t = t$, $t + r = r$.
Therefore, condition (ii) holds.

Finally, consider the case that $M_{14} = M'_{13}$, i.e, $\ms{v, s, t} = \ms{r + t, s, v}$.
A similar argument as in the previous case shows that $r + t = t$, $t + s = s$, $s + r = r$. Swapping the labels of the elements $r$ and $s$, we see that condition (ii) holds.
This completes the case analysis.

We have been working under the assumption that there is $y \in G$ such that $N_M(y)$ is not a multiple of $\binom{n-1}{2}$.
Now suppose that this is no longer so, i.e., $N_M(x)$ is a multiple of $\binom{n-1}{2}$ for all $x \in G$.
Then $\widetilde{M} = \widetilde{M}' = H \uplus H \uplus H$, where $H = M \uplus E = M' \uplus E'$ and $\card{E} = \card{E'} = 2$. If $M = \ms{m_1, m_2, m_3, m_4}$, then $E \uplus E \uplus E = \ms{m_{\min I} + m_{max I} : I \in \couples}$; similarly for $M'$ and $E'$. It thus holds that $x + y \in E$ whenever $\ms{x, y} \subseteq M$ and $x + y \in E'$ whenever $\ms{x, y} \subseteq M'$. Each one of the elements of $E$ arises in three different ways as a sum of two elements of $M$; each one of the elements of $E'$ arises in three different ways as a sum of two elements of $M'$.

We have several possibilities concerning the $6$-multiset $H$ and its possible partitions into a $4$-multiset $M$ and a $2$-multiset $E$ (which we will refer to as \emph{$(4,2)$-partitions} of $H$). The remainder of this proof is an analysis of the different cases that may arise. For easy reference, these cases are summarised in Table~\ref{table:types}, in which we also present the deck of each multiset $M$ considered. The different configurations $(M, E)$ will be referred to as ``types'', which are labeled with codes of the form $X.Y$, where $X$ is a Roman numeral and $Y$ is an Arabic numeral. We will also write simply ``$M$ is of type $X.Y$'' to mean ``$(M, E)$ is of type $X.Y$''.

Before starting the case analysis, let us first rule out an impossible configuration that may arise as a $(4,2)$-partition of $H$. If $(M, E) = (\ms{\alpha, \alpha, \beta, \beta}, \ms{\gamma, \delta})$ for some (not necessarily pairwise distinct) elements $\alpha, \beta, \gamma, \delta \in G$ such that $\gamma \neq \delta$, then $\ms{\alpha + \beta, \alpha, \beta}$ is a card of $M$ with multiplicity $4$ (if $\alpha \neq \beta$) or $6$ (if $\alpha = \beta$). But then $\alpha + \beta$ should be equal to both $\gamma$ and $\delta$, a contradiction which shows that this case does not occur.

\begin{table}
\begin{center}
\footnotesize
\setlength{\tabcolsep}{3.5pt}
\begin{tabular}{|c|cccccc|}
\hline
$H$ & \multicolumn{1}{c|}{$\ms{a, b, c, d, e, f}$} & \multicolumn{3}{c|}{$\ms{a, a, b, c, d, e}$} & \multicolumn{2}{c|}{$\ms{a, a, b, b, c, d}$}\\
type & \multicolumn{1}{c|}{I.1} & II.1 & II.2 & \multicolumn{1}{c|}{II.3} & III.1 & III.2 \\
$M$ & \multicolumn{1}{c|}{$\ms{\alpha, \beta, \gamma, \delta}$} & $\ms{a, a, \beta, \gamma}$ & $\ms{a, \beta, \gamma, \delta}$ & \multicolumn{1}{c|}{$\ms{b, c, d, e}$} & $\ms{\alpha, \alpha, \beta, \gamma}$ & $\ms{a, b, c, d}$ \\
$E$ & \multicolumn{1}{c|}{$\ms{\epsilon, \zeta}$} & $\ms{\delta, \epsilon}$ & $\ms{a, \epsilon}$ & \multicolumn{1}{c|}{$\ms{a, a}$} & $\ms{\beta, \delta}$ & $\ms{a, b}$ \\
\cline{2-7}
\multirow{6}{*}{deck} & \multicolumn{1}{c|}{$\ms{\alpha + \beta, \gamma, \delta}$} & $\ms{a + a, \beta, \gamma}$ & $\ms{a + \beta, \gamma, \delta}$ & \multicolumn{1}{c|}{$\ms{a, b, c}$} & $\ms{\alpha + \alpha, \beta, \gamma}$ & $\ms{a + b, c, d}$ \\
& \multicolumn{1}{c|}{$\ms{\alpha + \gamma, \beta, \delta}$} & $\ms{a + \beta, a, \gamma}$ & $\ms{a + \gamma, \beta, \delta}$ & \multicolumn{1}{c|}{$\ms{a, b, d}$} & $\ms{\beta + \gamma, \alpha, \alpha}$ & $\ms{a + c, b, d}$ \\
& \multicolumn{1}{c|}{$\ms{\alpha + \delta, \beta, \gamma}$} & $\ms{a + \beta, a, \gamma}$ & $\ms{a + \delta, \beta, \gamma}$ & \multicolumn{1}{c|}{$\ms{a, b, e}$} & $\ms{\alpha + \beta, \alpha, \gamma}$ & $\ms{a + d, b, c}$ \\
& \multicolumn{1}{c|}{$\ms{\beta + \gamma, \alpha, \delta}$} & $\ms{a + \gamma, a, \beta}$ & $\ms{\beta + \gamma, a, \delta}$ & \multicolumn{1}{c|}{$\ms{a, c, d}$} & $\ms{\alpha + \beta, \alpha, \gamma}$ & $\ms{b + c, a, d}$ \\
& \multicolumn{1}{c|}{$\ms{\beta + \delta, \alpha, \gamma}$} & $\ms{a + \gamma, a, \beta}$ & $\ms{\beta + \delta, a, \gamma}$ & \multicolumn{1}{c|}{$\ms{a, c, e}$} & $\ms{\alpha + \gamma, \alpha, \beta}$ & $\ms{b + d, a, c}$ \\
& \multicolumn{1}{c|}{$\ms{\gamma + \delta, \alpha, \beta}$} & $\ms{\beta + \gamma, a, a}$ & $\ms{\gamma + \delta, a, \beta}$ & \multicolumn{1}{c|}{$\ms{a, d, e}$} & $\ms{\alpha + \gamma, \alpha, \beta}$ & $\ms{c + d, a, b}$ \\
\hline
$H$ & \multicolumn{3}{c|}{$\ms{a, a, a, b, c, d}$} & \multicolumn{2}{c|}{$\ms{a, a, b, b, c, c}$} & impossible \\
type & IV.1 & IV.2 & \multicolumn{1}{c|}{IV.3} & V.1 & \multicolumn{1}{c|}{V.2} & configuration \\
$M$ & $\ms{a, a, a, \beta}$ & $\ms{a, a, \beta, \gamma}$ & \multicolumn{1}{c|}{$\ms{a, b, c, d}$} & $\ms{\alpha, \alpha, \beta, \beta}$ & \multicolumn{1}{c|}{$\ms{\alpha, \alpha, \beta, \gamma}$} & $\ms{\alpha, \alpha, \beta, \beta}$ \\
$E$ & $\ms{\gamma, \delta}$ & $\ms{a, \delta}$ & \multicolumn{1}{c|}{$\ms{a, a}$} & $\ms{\gamma, \gamma}$ & \multicolumn{1}{c|}{$\ms{\beta, \gamma}$} & $\ms{\gamma, \delta}$ \\
\cline{2-7}
\multirow{6}{*}{deck} & $\ms{a + a, a, \beta}$ & $\ms{a + a, \beta, \gamma}$ & \multicolumn{1}{c|}{$\ms{a, a, b}$} & $\ms{\gamma, \alpha, \alpha}$ & \multicolumn{1}{c|}{$\ms{\alpha + \alpha, \beta, \gamma}$} & $\ms{\alpha + \alpha, \beta, \beta}$ \\
& $\ms{a + a, a, \beta}$ & $\ms{a + \beta, a, \gamma}$ & \multicolumn{1}{c|}{$\ms{a, a, c}$} & $\ms{\gamma, \beta, \beta}$ & \multicolumn{1}{c|}{$\ms{\beta + \gamma, \alpha, \alpha}$} & $\ms{\beta + \beta, \alpha, \alpha}$  \\
& $\ms{a + a, a, \beta}$ & $\ms{a + \beta, a, \gamma}$ & \multicolumn{1}{c|}{$\ms{a, a, d}$} & $\ms{\gamma, \alpha, \beta}$ & \multicolumn{1}{c|}{$\ms{\alpha + \beta, \alpha, \gamma}$} & $\ms{\alpha + \beta, \alpha, \beta}$ \\
& $\ms{a + \beta, a, a}$ & $\ms{a + \gamma, a, \beta}$ & \multicolumn{1}{c|}{$\ms{a, b, c}$} & $\ms{\gamma, \alpha, \beta}$ & \multicolumn{1}{c|}{$\ms{\alpha + \beta, \alpha, \gamma}$} & $\ms{\alpha + \beta, \alpha, \beta}$ \\
& $\ms{a + \beta, a, a}$ & $\ms{a + \gamma, a, \beta}$ & \multicolumn{1}{c|}{$\ms{a, b, d}$} & $\ms{\gamma, \alpha, \beta}$ & \multicolumn{1}{c|}{$\ms{\alpha + \gamma, \alpha, \beta}$} & $\ms{\alpha + \beta, \alpha, \beta}$ \\
& $\ms{a + \beta, a, a}$ & $\ms{\beta + \gamma, a, a}$ & \multicolumn{1}{c|}{$\ms{a, c, d}$} & $\ms{\gamma, \alpha, \beta}$ & \multicolumn{1}{c|}{$\ms{\alpha + \gamma, \alpha, \beta}$} & $\ms{\alpha + \beta, \alpha, \beta}$ \\
\hline
$H$ & \multicolumn{4}{c|}{$\ms{a, a, a, b, b, c}$} & \multicolumn{2}{c|}{$\ms{a, a, a, a, b, c}$} \\
type & VI.1 & VI.2 & VI.3 & \multicolumn{1}{c|}{VI.4} & VII.1 & VII.2 \\
$M$ & $\ms{a, a, a, b}$ & $\ms{a, a, a, c}$ & $\ms{a, a, b, c}$ & \multicolumn{1}{c|}{$\ms{a, b, b, c}$} & $\ms{a, a, a, \beta}$ & $\ms{a, a, b, c}$ \\
$E$ & $\ms{b, c}$ & $\ms{b, b}$ & $\ms{a, b}$ & \multicolumn{1}{c|}{$\ms{a, a}$} & $\ms{a, \gamma}$ & $\ms{a, a}$ \\
\cline{2-7}
\multirow{6}{*}{deck} & $\ms{a + a, a, b}$ & $\ms{b, a, c}$ & $\ms{a + a, b, c}$ & \multicolumn{1}{c|}{$\ms{a, a, c}$} & $\ms{a + a, a, \beta}$ & $\ms{a, a, a}$ \\
& $\ms{a + a, a, b}$ & $\ms{b, a, c}$ & $\ms{b + c, a, a}$ & \multicolumn{1}{c|}{$\ms{a, b, b}$} & $\ms{a + a, a, \beta}$ & $\ms{a, b, c}$ \\
& $\ms{a + a, a, b}$ & $\ms{b, a, c}$ & $\ms{a + b, a, c}$ & \multicolumn{1}{c|}{$\ms{a, a, b}$} & $\ms{a + a, a, \beta}$ & $\ms{a, a, b}$ \\
& $\ms{a + b, a, a}$ & $\ms{b, a, a}$ & $\ms{a + b, a, c}$ & \multicolumn{1}{c|}{$\ms{a, a, b}$} & $\ms{a + \beta, a, a}$ & $\ms{a, a, b}$ \\
& $\ms{a + b, a, a}$ & $\ms{b, a, a}$ & $\ms{a + c, a, b}$ & \multicolumn{1}{c|}{$\ms{a, b, c}$} & $\ms{a + \beta, a, a}$ & $\ms{a, a, c}$ \\
& $\ms{a + b, a, a}$ & $\ms{b, a, a}$ & $\ms{a + c, a, b}$ & \multicolumn{1}{c|}{$\ms{a, b, c}$} & $\ms{a + \beta, a, a}$ & $\ms{a, a, c}$ \\
\hline
$H$ & \multicolumn{1}{c|}{$\ms{a, a, a, b, b, b}$} & \multicolumn{3}{c|}{$\ms{a, a, a, a, b, b}$} & \multicolumn{1}{c|}{$\ms{a, a, a, a, a, b}$} & $\ms{a, a, a, a, a, a}$ \\
type & \multicolumn{1}{c|}{VIII.1} & IX.1 & IX.2 & \multicolumn{1}{c|}{IX.3} & \multicolumn{1}{c|}{X.1} & XI.1 \\
$M$ & \multicolumn{1}{c|}{$\ms{\alpha, \alpha, \alpha, \beta}$} & $\ms{a, a, a, a}$ & $\ms{a, a, a, b}$ & \multicolumn{1}{c|}{$\ms{a, a, b, b}$} & \multicolumn{1}{c|}{$\ms{a, a, a, b}$} & $\ms{a, a, a, a}$ \\
$E$ & \multicolumn{1}{c|}{$\ms{\beta, \beta}$} & $\ms{b, b}$ & $\ms{a, b}$ & \multicolumn{1}{c|}{$\ms{a, a}$} & \multicolumn{1}{c|}{$\ms{a, a}$} & $\ms{a, a}$ \\
\cline{2-7}
\multirow{6}{*}{deck} & \multicolumn{1}{c|}{$\ms{\beta, \alpha, \alpha}$} & $\ms{b, a, a}$ & $\ms{a + a, a, b}$ & \multicolumn{1}{c|}{$\ms{a, a, a}$} & \multicolumn{1}{c|}{$\ms{a, a, b}$} & $\ms{a, a, a}$ \\
& \multicolumn{1}{c|}{$\ms{\beta, \alpha, \alpha}$} & $\ms{b, a, a}$ & $\ms{a + a, a, b}$ & \multicolumn{1}{c|}{$\ms{a, b, b}$} & \multicolumn{1}{c|}{$\ms{a, a, b}$} & $\ms{a, a, a}$ \\
& \multicolumn{1}{c|}{$\ms{\beta, \alpha, \alpha}$} & $\ms{b, a, a}$ & $\ms{a + a, a, b}$ & \multicolumn{1}{c|}{$\ms{a, a, b}$} & \multicolumn{1}{c|}{$\ms{a, a, b}$} & $\ms{a, a, a}$ \\
& \multicolumn{1}{c|}{$\ms{\beta, \alpha, \beta}$} & $\ms{b, a, a}$ & $\ms{a + b, a, a}$ & \multicolumn{1}{c|}{$\ms{a, a, b}$} & \multicolumn{1}{c|}{$\ms{a, a, a}$} & $\ms{a, a, a}$ \\
& \multicolumn{1}{c|}{$\ms{\beta, \alpha, \beta}$} & $\ms{b, a, a}$ & $\ms{a + b, a, a}$ & \multicolumn{1}{c|}{$\ms{a, a, b}$} & \multicolumn{1}{c|}{$\ms{a, a, a}$} & $\ms{a, a, a}$ \\
& \multicolumn{1}{c|}{$\ms{\beta, \alpha, \beta}$} & $\ms{b, a, a}$ & $\ms{a + b, a, a}$ & \multicolumn{1}{c|}{$\ms{a, a, b}$} & \multicolumn{1}{c|}{$\ms{a, a, a}$} & $\ms{a, a, a}$ \\
\hline
\end{tabular}
\end{center}
\caption{The different types of multisets considered in the proof of Theorem~\ref{thm:neq4}.}
\label{table:types}
\end{table}

\textit{Case 1:} $H = \ms{a, b, c, d, e, f}$ for some pairwise distinct elements $a, b, c, d, e, f \in G$.
The $(4,2)$-partitions of $H$ are of the form $(\ms{\alpha, \beta, \gamma, \delta}, \ms{\epsilon, \zeta})$ (referred to as type I.1), where $\{\alpha, \beta, \gamma, \delta, \epsilon, \zeta\} = \{a, b, c, d, e, f\}$.
Thus, we may assume that $M = \ms{\alpha, \beta, \gamma, \delta}$, $E = \ms{\epsilon, \zeta}$, $M' = \ms{\alpha', \beta', \gamma', \delta'}$, $E' = \ms{\epsilon', \zeta'}$ with $\{\alpha, \beta, \gamma, \delta, \epsilon, \zeta\} = \{\alpha', \beta', \gamma', \delta', \epsilon', \zeta'\} = \{a, b, c, d, e, f\}$.

Suppose, on the contrary, that $M \neq M'$. Then $2 \leq \card{M \cap M'} \leq 3$.
If $\card{M \cap M'} = 2$, then we may assume, without loss of generality, that $M' = \ms{\alpha, \beta, \epsilon, \zeta}$. Then $\alpha + \beta \in E \cap E' = \ms{\epsilon, \zeta} \cap \ms{\gamma, \delta} = \emptyset$, a contradiction.
If $\card{M \cap M'} = 3$, then we may assume, without loss of generality, that $M' = \ms{\alpha, \beta, \gamma, \epsilon}$. Then $\alpha + \beta \in E \cap E' = \ms{\epsilon, \zeta} \cap \ms{\delta, \zeta} = \ms{\zeta}$; thus $\alpha + \beta = \zeta$. It follows that $\ms{\zeta, \gamma, \epsilon}$ is a card of $M'$, but this cannot be a card of $M$, because $\epsilon$ and $\zeta$ do not occur together in any card of $M$. We have reached again a contradiction. We conclude that $M = M'$.

\textit{Case 2:} $H = \ms{a, a, b, c, d, e}$ for some pairwise distinct elements $a, b, c, d, e \in G$.
The $(4,2)$-partitions of $H$ are the following:
$(\ms{a, a, \beta, \gamma}, \ms{\delta, \epsilon})$ (type II.1),
$(\ms{a, \beta, \gamma, \delta}, \ms{a, \epsilon})$ (type II.2), and
$(\ms{b, c, d, e}, \ms{a, a})$ (type II.3),
where $\{\beta, \gamma, \delta, \epsilon\} = \{b, c, d, e\}$.

Consider first the case that $M$ is of type II.1 and $M'$ is of type II.2 or II.3. Then the deck of $M$ has repeated cards while the cards of $M'$ are all pairwise distinct; hence $\deck M \neq \deck M'$. We have reached a contradiction, which shows that this case is not possible.

Consider then the case that $M = \ms{a, \beta, \gamma, \delta}$ is of type II.2 and $M' = \ms{b, c, d, e}$ is of type II.3\@. Every card of $M'$ contains exactly one occurrence of $a$. In order to have exactly one occurrence of $a$ in every card of $M$, we must have $a + \beta = a + \gamma = a + \delta = a$; consequently $\beta + \gamma = \beta + \delta = \gamma + \delta = \epsilon$. The fact that $M'$ is of type 2.C implies $x + y = a$ for all distinct $x, y \in \{b, c, d, e\}$. We have reached a contradiction, which shows that this case is not possible.

Consider then the case that $M = \ms{a, a, \beta, \gamma}$ and $M' = \ms{a, a, \beta', \gamma'}$ are both of type II.1, with $\{\beta, \gamma, \delta, \epsilon\} = \{\beta', \gamma', \delta', \epsilon'\} = \{b, c, d, e\}$.
The only card of $M$ with no occurrence of $a$ is $\ms{a + a, \beta, \gamma}$, and the only card of $M'$ with no occurrence of $a$ is $\ms{a + a, \beta', \gamma'}$. These must be equal; hence $\{\beta, \gamma\} = \{\beta', \gamma'\}$, that is, $M = M'$.

Consider then the case that $M$ and $M'$ are both of type II.2. Suppose, on the contrary, that $M \neq M'$. We may assume, without loss of generality, that $M = \ms{a, \beta, \gamma, \delta}$, $M' = \ms{a, \beta, \gamma, \epsilon}$, where $\{\beta, \gamma, \delta, \epsilon\} = \{b, c, d, e\}$.
Then $\beta + \gamma \in E \cap E' = \ms{a, \epsilon} \cap \ms{a, \delta} = \ms{a}$.
Then $\ms{a, a, \delta}$ is a card of $M$, but this is not a card of $M'$. We have arrived in a contradiction, and we conclude that $M = M'$.

Finally, if $M$ and $M'$ are both of type II.3, then $M = M'$.

\textit{Case 3:}
$H = \ms{a, a, b, b, c, d}$ for some pairwise distinct elements $a, b, c, d \in G$.
The possible $(4,2)$-partitions of $H$ are the following:
$(\ms{\alpha, \alpha, \beta, \gamma}, \ms{\beta, \delta}$ (type III.1) and
$(\ms{a, b, c, d}, \ms{a, b})$ (type III.2),
where $\{\alpha, \beta\} = \{a, b\}$ and $\{\gamma, \delta\} = \{c, d\}$.
(The $(4,2)$-partition $(\ms{a, a, b, b}, \ms{c, d})$ of $H$ is not possible, as noted above.)

Consider first the case that $M = \ms{\alpha, \alpha, \beta, \gamma}$ is of type III.1 and $M' = \ms{a, b, c, d}$ is of type III.2\@. Then $\alpha + \beta, \alpha + \gamma, \beta + \gamma \in E \cap E' = \ms{\beta, \delta} \cap \ms{a, b} = \ms{\beta}$, but then $\delta$ would appear only at most once in the cards of $M$. We have arrived in a contradiction, which shows that this case is not possible.

Consider then the case that $M = \ms{\alpha, \alpha, \beta, \gamma}$ and $M' = \ms{\alpha', \alpha', \beta', \gamma'}$ are both of type III.1, with $\{\alpha, \beta\} = \{\alpha', \beta'\} = \{a, b\}$, $\{\gamma, \delta\} = \{\gamma', \delta'\} = \{c, d\}$.
If $\alpha' = \beta$ and $\gamma' = \delta$, then $\alpha + \beta = \alpha' + \beta' \in E \cap E' = \ms{\beta, \delta} \cap \ms{\alpha, \gamma} = \emptyset$, a contradiction.
If $\alpha' = \alpha$ and $\gamma' = \delta$, then $\alpha + \alpha, \alpha + \beta \in E \cap E' = \ms{\beta, \delta} \cap \ms{\beta, \gamma} = \ms{\beta}$; consequently, $\ms{\beta, \beta, \gamma}$ is a card of $M$ but not a card of $M'$, a contradiction.
If $\alpha' = \beta$ and $\gamma' = \gamma$, then $\alpha + \beta, \beta + \gamma \in E \cap E' = \ms{\beta, \delta} \cap \ms{\alpha, \delta} = \ms{\delta}$; but then $\delta$ will occur at least $4$ times in the cards of $M'$, a contradiction.
We are left with the case that $\alpha' = \alpha$ and $\gamma' = \gamma$, that is, $M = M'$.

Finally, if $M$ and $M'$ are both of the form III.2, then $M = M'$.

\textit{Case 4:}
$H = \ms{a, a, a, b, c, d}$ for some pairwise distinct elements $a, b, c, d \in G$.
The $(4,2)$-partitions of $H$ are the following:
$(\ms{a, a, a, \beta}, \ms{\gamma, \delta})$ (type IV.1),
$(\ms{a, a, \beta, \gamma}, \ms{a, \delta})$ (type IV.2), and
$(\ms{a, b, c, d}, \ms{a, a})$ (type IV.3),
where $\{\beta, \gamma, \delta\} = \{b, c, d\}$.

Consider first the case that $M$ is of type IV.1 or IV.2 and $M'$ is of type IV.3. Then $M$ has repeated cards while the cards of $M'$ are pairwise distinct; hence $\deck M \neq \deck M$. We have reached a contradiction, which shows that this case is not possible.

Consider then the case that $M = \ms{a, a, a, \beta}$ is of type IV.1 and $M' = \ms{a, a, \beta', \gamma'}$ is of type IV.2, with $\{\beta, \gamma, \delta\} = \{\beta', \gamma', \delta'\} = \{b, c, d\}$. If $\{\beta\} \cap \{\beta', \gamma'\} = \emptyset$, then actually $\{\beta', \gamma'\} = \{\gamma, \delta\}$ and $\delta' = \beta$. Then we would have $a + a \in E \cap E' = \ms{\gamma, \delta} \cap \ms{a, \beta} = \emptyset$, a contradiction. Thus, we may assume that $\beta \in \{\beta', \gamma'\}$. Then we have that $a + a, a + \beta \in E \cap E' \ms{\gamma, \delta} \cap \ms{a, \delta'} = \ms{\delta'}$. On the other hand, it follows from the fact that $M$ is of type IV.1 that $a + a \neq a + \beta$. We have reached again a contradiction, and we conclude that this case is not possible.

Consider then the case that $M = \ms{a, a, a, \beta}$ and $M' = \ms{a, a, a, \beta'}$ are both of type IV.1, with $\{\beta, \gamma, \delta\} = \{\beta', \gamma', \delta'\} = \{b, c, d\}$. Since $a + a \in E \cap E' = \ms{\gamma, \delta} \cap \ms{\gamma', \delta'}$, we have that $a + a \neq a$. Therefore, the only cards of $M$ with a single occurrence of $a$ are the three copies of $\ms{a + a, a, \beta}$, and the only cards of $M'$ with a single occurrence of $a$ are the three copies of $\ms{a + a, a, \beta'}$. This implies that $\beta = \beta'$; hence $M = M'$.

Consider then the case that $M = \ms{a, a, \beta, \gamma}$ and $M' = \ms{a, a, \beta', \gamma'}$ are both of type IV.2, with $\{\beta, \gamma, \delta\} = \{\beta', \gamma', \delta'\} = \{b, c, d\}$. Suppose, on the contrary, that $M \neq M'$. We may assume, without loss of generality, that $\beta = \beta'$ and $\gamma' = \delta$. Then $a + a, a + \beta \in E \cap E' = \ms{a, \delta} \cap \ms{a, \gamma} = \ms{a}$. Consequently, $\ms{a, \beta, \gamma}$ has multiplicity exactly $1$ in the deck of $M$ but it has multiplicity $2$ in the deck of $M'$, a contradiction. We conclude that $M = M'$.

Finally, if $M$ and $M'$ are both of type IV.3, then $M = M'$.

\textit{Case 5:}
$H = \ms{a, a, b, b, c, c}$ for some pairwise distinct elements $a, b, c \in G$.
The $(4,2)$-partitions of $H$ are the following:
$(\ms{\alpha, \alpha, \beta, \beta}, \ms{\gamma, \gamma})$ (type V.1) and
$(\ms{\alpha, \alpha, \beta, \gamma}, \ms{\beta, \gamma})$ (type V.2),
where $\{\alpha, \beta, \gamma\} = \{a, b, c\}$.

Consider first the case that $M = \ms{\alpha, \alpha, \beta, \beta}$ is of type V.1 and $M' = \linebreak[4] \ms{\alpha', \alpha', \beta', \gamma'}$ is of type V.2, with $\{\alpha, \beta, \gamma\} = \{\alpha', \beta', \gamma'\} = \{a, b, c\}$. The element $\gamma$ occurs exactly once in every card of $M$. One of the cards of $M'$, namely $\ms{\beta' + \gamma', \alpha', \alpha'}$, has two occurrences of $\alpha'$; hence $\alpha' \neq \gamma$. Suppose $\beta' = \gamma$; in other words, $\{\alpha', \gamma'\} = \{\alpha, \beta\}$. Then we must have $\beta' + \gamma' = \alpha' + \beta' = \gamma = \beta'$; consequently, $\alpha' + \alpha' = \alpha' + \gamma' = \gamma' \neq \gamma$, but this contradicts the fact that $\alpha + \alpha = \alpha + \beta = \beta + \beta = \gamma$, implied by the fact that $M$ is of type V.1\@. Thus, we remain with the possibility that $\gamma' = \gamma$; in other words, $\{\alpha', \beta'\} = \{\alpha, \beta\}$. Then $\beta' + \gamma' = \alpha' + \gamma' = \gamma = \gamma'$; consequently, $\alpha' + \alpha' = \alpha' + \beta' = \beta' \neq \gamma$, and we arrive similarly in a contradiction. We conclude that this case is not possible.

Consider then the case that $M = \ms{\alpha, \alpha, \beta, \beta}$ and $M' = \ms{\alpha', \alpha', \beta', \beta'}$ are both of type V.1, with $\{\alpha, \beta, \gamma\} = \{\alpha', \beta', \gamma'\} = \{a, b, c\}$. Then the unique element occurring exactly once in every card of $M$ is $\gamma$, and the unique element occurring exactly once in every card of $M'$ is $\gamma'$. Hence $\gamma = \gamma'$, that is, $M = M'$.

Consider finally the case that $M = \ms{\alpha, \alpha, \beta, \gamma}$ and $M' = \ms{\alpha', \alpha', \beta', \gamma'}$ are both of type V.2, with $\{\alpha, \beta, \gamma\} = \{\alpha', \beta', \gamma'\} = \{a, b, c\}$. Suppose, on the contrary, that $M \neq M'$. We may assume, without loss of generality, that $M' = \ms{\beta, \beta, \alpha, \gamma}$. Then we have that $\alpha + \beta, \alpha + \gamma \in E \cap E' = \ms{\beta, \gamma} \cap \ms{\alpha, \gamma} = \ms{\gamma}$, and we will have too many $\gamma$'s occurring in the cards of $M$, a contradiction. We conclude that $M = M'$.

\textit{Case 6:}
$H = \ms{a, a, a, b, b, c}$ for some pairwise distinct elements $a, b, c \in G$.
The possible $(4,2)$-partitions of $H$ are the following:
$(\ms{a, a, a, b}, \ms{b, c})$ (type VI.1),
$(\ms{a, a, a, c}, \ms{b, b})$ (type VI.2),
$(\ms{a, a, b, c}, \ms{a, b})$ (type VI.3), and
$(\ms{a, b, b, c}, \ms{a, a})$ (type VI.4).
(The $(4,2)$-partition $(\ms{a, a, b, b}, \ms{a, c})$ of $H$ is not possible, as noted above.)

If $M$ and $M'$ are of the same type, VI.1, VI.2, VI.3, or VI.4, then clearly $M = M'$. Suppose then, on the contrary, that $M$ and $M'$ are of different types.

Assume that $M$ is of type VI.4\@. Then $a + b = a + c = b + c = b + b = a$. If $M'$ is of type VI.1, then $a + b \neq a$, a contradiction. If $M'$ is of type VI.2, then $a + a = a + c = b$, a contradiction. If $M'$ is of type VI.3, then $a + b$ and $a + c$ are not both equal to $a$, a contradiction.

Assume that $M$ is of type VI.2. Then $a + a = b$. If $M'$ is of type VI.1, then $\ms{b, a, b}$ is a card of $M'$ but it is not a card of $M$, a contradiction. If $M'$ is of type VI.3, then $\ms{b, b, c}$ is a card of $M'$ but it is not a card of $M$, a contradiction.

Assume that $M$ is of type VI.1 and $M'$ is of type VI.3\@. The fact that $M$ is of type VI.1 implies $\{a + a, a + b\} = \{b, c\}$. The fact that $M'$ is of type VI.3 implies $\{a + a, a + b\} \subseteq \{a, b\}$, a contradiction.

\textit{Case 7:}
$H = \ms{a, a, a, a, b, c}$ for some pairwise distinct elements $a, b, c \in G$.
The possible $(4,2)$-partitions of $H$ are the following:
$(\ms{a, a, a, \beta}, \ms{a, \gamma})$ (type VII.1) and
$(\ms{a, a, b, c}, \ms{a, a})$ (type VII.2),
where $\{\beta, \gamma\} = \{b, c\}$.
(The $(4,2)$-partition $(\ms{a, a, a, a}, \ms{b, c})$ of $G$ is not possible, as noted above.)

Consider first the case that $M$ is of type VII.1 and $M'$ is of type VII.2\@. Then $\ms{a, a, a}$ has multiplicity $1$ in the deck of $M'$, but its multiplicity is either $0$ or $3$ in the deck of $M$, a contradiction. We conclude that this case is not possible.

Consider then the case that $M$ and $M'$ are both of type VII.1. If $M = M'$, then we are done. Assume that $M \neq M$. We may assume that $M = \ms{a, a, a, b}$ and $M' = \ms{a, a, a, c}$. Then $a + a \in E \cap E' = \ms{a, c} \cap \ms{a, b} = \ms{a}$. This implies that $a + b = c$ and $a + c = b$.
Choosing $r := a$, $s := a$, $t := a$, $u := b$, $v := c$, we see that condition (ii) holds.

Finally, if $M$ and $M'$ are both of type VII.2, then $M = M'$.

\textit{Case 8:}
$H = \ms{a, a, a, b, b, b}$ for some distinct elements $a, b \in G$.
The $(4,2)$-partitions of $H$ are of the form
$(\ms{\alpha, \alpha, \alpha, \beta}, \ms{\beta, \beta})$ (type VIII.1),
where $\{\alpha, \beta\} = \{a, b\}$.
(The $(4,2)$-partition $(\ms{a, a, b, b}, \ms{a, b})$ of $G$ is not possible, as noted above.)

Suppose, on the contrary, that $M \neq M'$. We may assume that $M = \ms{a, a, a, b}$ and $M' = \ms{b, b, b, a}$. Then $a + a = a + b = b$ and $b + b = b + a = a$, a contradiction. We conclude that $M = M'$.

\textit{Case 9:}
$H = \ms{a, a, a, a, b, b}$ for some distinct elements $a, b \in G$.
The $(4,2)$-partitions of $H$ are the following:
$(\ms{a, a, a, a}, \ms{b, b})$ (type IX.1),
$(\ms{a, a, a, b}, \ms{a, b})$ (type IX.2), and
$(\ms{a, a, b, b}, \ms{a, a})$ (type IX.3).

If $M$ and $M'$ are of the same type, IX.1, IX.2, or IX.3, then clearly $M = M'$. Suppose, on the contrary, that $M$ and $M'$ are of different types.

Assume that $M$ is of type IX.1. Then $a + a = b$. If $M'$ is of type IX.2, then $\ms{b, a, b}$ is a card of $M'$, but this is not a card of $M$, a contradiction. If $M'$ is of type IX.3, then $a + a = a$, a contradiction.

Assume that $M$ is of type IX.2. Then $a + a \neq a + b$. If $M'$ is of type IX.3, then $a + a = a + b = b + b = a$, a contradiction.

\textit{Case 10:}
$H = \ms{a, a, a, a, a, b}$ for some distinct elements $a, b \in G$.
The only possible $(4,2)$-partition of $H$ is
$(\ms{a, a, a, b}, \ms{a, a})$ (type X.1).
(The $(4,2)$-partition $(\ms{a, a, a, a}, \ms{a, b})$ of $H$ is not possible, as noted above.)
Therefore, $M = M'$.

\textit{Case 11:}
$H = \ms{a, a, a, a, a, a}$ for some $a \in G$.
The only $(4,2)$-partition of $H$ is $(\ms{a, a, a, a}, \ms{a, a})$ (type XI.1),
and it holds that $M = M'$.

We have exhausted all possible cases, and have arrived at the desired conclusion. This completes the proof of the theorem.
\end{proof}

\begin{theorem}
Assume that $(G; +)$ is a commutative groupoid. Let $M$ and $M'$ be multisets of cardinality $3$ over $G$. Then $\deck M = \deck M'$ if and only if one of the following conditions holds:
\begin{enumerate}[\rm (i)]
\item $M = M'$.
\item $M = \ms{r, s, t}$, $M' = \ms{r, r + s, r + t}$ for some elements $r, s, t \in G$ satisfying $r + (r + s) = s$, $r + (r + t) = t$, $(r + s) + (r + t) = s + t$.
\item $M = \ms{r, s, t}$, $M' = \ms{r + s, r + t, s + t}$ for some elements $r, s, t \in G$ satisfying $(r + s) + (r + t) = r$, $(r + s) + (s + t) = s$, $(r + t) + (s + t) = t$.
\end{enumerate}
\end{theorem}

\begin{proof}
It is clear that if $M = M'$, then $\deck M = \deck M'$.
If condition (ii) or (iii) holds, then $\deck M = \deck M'$, as shown in Examples~\ref{ex:3abaceqbc} and~\ref{ex:3abaceqa}.

For the converse implication, assume that $\deck M = \deck M'$.
Assume that $M = \ms{a, b, c}$ and $M' = \ms{\alpha, \beta, \gamma}$. Then
\begin{align*}
\deck M &=
\ms{ \ms{a, b + c}, \ms{b, a + c}, \ms{c, a + b} }, \\
\deck M' &=
\ms{ \ms{\alpha, \beta + \gamma}, \ms{\beta, \alpha + \gamma}, \ms{\gamma, \alpha + \beta} }.
\end{align*}
Relabeling the elements of $M'$ if necessary, we may assume that
\[
\ms{\alpha, \beta + \gamma} = \ms{a, b + c}, \quad
\ms{\beta, \alpha + \gamma} = \ms{b, a + c}, \quad
\ms{\gamma, \alpha + \beta} = \ms{c, a + b}.
\]

If $(\alpha, \beta + \gamma) = (a, b + c)$, $(\beta, \alpha + \gamma) = (b, a + c)$, $(\gamma, \alpha + \beta) = (c, a + b)$, then $M = M'$ and we are done.

If $(\alpha, \beta + \gamma) = (a, b + c)$, $(\beta, \alpha + \gamma) = (b, a + c)$, $(\gamma, \alpha + \beta) = (a + b, c)$, then we have $c = \alpha + \beta = a + b = \gamma$. Hence $M = M'$ and we are done.
If $(\alpha, \beta + \gamma) = (a, b + c)$, $(\beta, \alpha + \gamma) = (a + c, b)$, $(\gamma, \alpha + \beta) = (c, a + b)$ or $(\alpha, \beta + \gamma) = (b + c, a)$, $(\beta, \alpha + \gamma) = (b, a + c)$, $(\gamma, \alpha + \beta) = (c, a + b)$, then a similar argument shows that $M = M'$ and we are done.

If $(\alpha, \beta + \gamma) = (a, b + c)$, $(\beta, \alpha + \gamma) = (a + c, b)$, $(\gamma, \alpha + \beta) = (a + b, c)$, then $a + (a + b) = \alpha + \gamma = b$, $a + (a + c) = \alpha + \beta = c$ and $(a + b) + (a + c) = \gamma + \beta = b + c$.
Choosing $r := a$, $s := b$, $t := c$, we see that condition (ii) holds and we are done.
We argue similarly in the case when $(\alpha, \beta + \gamma) = (b + c, a)$, $(\beta, \alpha + \gamma) = (b, a + c)$, $(\gamma, \alpha + \beta) = (a + b, c)$ or $(\alpha, \beta + \gamma) = (b + c, a)$, $(\beta, \alpha + \gamma) = (a + c, b)$, $(\gamma, \alpha + \beta) = (c, a + b)$ to show that condition (ii) holds.

We are left with the case that $(\alpha, \beta + \gamma) = (b + c, a)$, $(\beta, \alpha + \gamma) = (a + c, b)$, $(\gamma, \alpha + \beta) = (a + b, c)$. Then $(a + b) + (a + c) = \gamma + \beta = a$, $(a + b) + (b + c) = \gamma + \alpha = b$ and $(a + c) + (b + c) = \beta + \alpha = c$. Choosing $r := a$, $s := b$, $t := c$, we see that condition (iii) holds.
\end{proof}

\begin{theorem}
Assume that $(G; +)$ is a commutative groupoid. Let $M$ and $M'$ be multisets of cardinality $2$ over $G$. Then $\deck M = \deck M'$ if and only if $M = \ms{r, s}$, $M' = \ms{t, u}$ for some elements $r, s, t, u \in G$ such that $r + s = t + u$.
\end{theorem}

\begin{proof}
Obvious.
\end{proof}

\begin{remark}
Let $(G; +)$ be a commutative groupoid. Every multiset of cardinality $2$ over $G$ is reconstructible if and only if for all $a, b, c, d \in G$, it holds that
\[
a + b = c + d
\iff
(a, b) = (c, d) \text{~or~} (a, b) = (d, c).
\]
Examples of groupoids satisfying this condition include the free commutative groupoids (see the paper by Pre\v{s}i\'{c}~\cite{Presic}).
\end{remark}

\subsection{Open problems}

We have completely solved the reconstruction problem for multisets over commutative groupoids. We conclude this section by suggesting some possible directions for future research. A common variant of reconstruction problems is the so-called set reconstruction problem: we define deck as a \emph{set} of cards instead of a \emph{multiset} of cards and then ask whether an object is uniquely determined (up to isomorphism) by its set of cards. The set reconstruction problem for multisets over commutative groupoids is an open problem, and it may be worth investigating.

Another related question is the following: Is reconstruction possible from a few cards only? More precisely, for a commutative groupoid $(G; +)$ and an integer $n \geq 2$, what is the smallest number $m$ such that every multiset $M$ of cardinality $n$ over $G$ is uniquely determined by any $m$ of its cards? This remains an open problem, but let us make a few simple observations. This number may be as large as $\binom{n}{2}$, i.e., all cards are needed for reconstruction, as the following example illustrates.

\begin{example}
Let $(G; +)$ be the $2$-element group of addition modulo $2$, and let $n = 4$. Let $M = \ms{1, 1, 1, 1}$, $M' = \ms{0, 0, 1, 1}$. Then
\begin{align*}
\deck M &= \ms{\ms{0, 1, 1}, \ms{0, 1, 1}, \ms{0, 1, 1}, \ms{0, 1, 1}, \ms{0, 1, 1}, \ms{0, 1, 1}}, \\
\deck M &= \ms{\ms{0, 1, 1}, \ms{0, 1, 1}, \ms{0, 1, 1}, \ms{0, 1, 1}, \ms{0, 1, 1}, \ms{0, 0, 0}}.
\end{align*}
Even though every $4$-multiset over $G$ is reconstructible by Theorem~\ref{thm:neq4}, $M$ and $M'$ cannot be reconstructed from $5$ cards only: the decks of both $M$ and $M'$ include $5$ copies of $\ms{0, 1, 1}$.
\end{example}

If $(G; +)$ is a commutative group, then $2$ cards do not suffice for reconstruction.
This clearly holds for $2$-multisets, and the following example shows that this is also the case for multisets of cardinality at least $3$.

\begin{example}
Let $(G; +)$ be a commutative group and assume that $n \geq 3$. Let $M = \ms{m_1, \dots, m_n}$, $M' = \ms{m_1 + m_2, m_2 + m_3, -m_2, m_4, \dots, m_n}$. Then
\[
\ms{m_1 + m_2, m_3, m_4, \dots, m_n}
\quad \text{and} \quad
\ms{m_1, m_2 + m_3, m_4, \dots, m_n}
\]
are cards of both $M$ and $M'$.
\end{example}

%%%%%%%%%%%%%%%%%%%%%%%%%%%%%%%%%%%%%%%%%%%%%%%%%%

\section{Reconstruction problem for functions of several arguments -- the case of affine functions}

As mentioned in the introduction, the reconstruction problem for multisets over commutative groupoids arose from a completely different reconstruction problem formulated for functions of several arguments.
In the special case of affine functions over nonassociative semirings, the reconstruction problem for functions reduces to the reconstruction problem for multisets over commutative groupoids.
In this section, we will apply our results on the reconstructibility of multisets to the reconstruction problem for functions.

\subsection{Functions of several arguments and identification minors}

Let $A$ and $B$ be arbitrary sets with at least two elements.
A \emph{function} (\emph{of several arguments}) from $A$ to $B$ is a map $f \colon A^n \to B$ for some positive integer $n$, called the \emph{arity} of $f$. Functions of several arguments from $A$ to $A$ are called \emph{operations} on $A$.
We denote the set of all $n$-ary functions from $A$ to $B$ by $\cl{F}_{AB}^{(n)}$, and we denote the set of all functions from $A$ to $B$ of any finite arity by $\cl{F}_{AB}$; in other words, $\cl{F}_{AB}^{(n)} = B^{A^n}$ and $\cl{F}_{AB} = \bigcup_{n \geq 1} \cl{F}_{AB}^{(n)}$.
For $1 \leq i \leq n$, the $i$-th $n$-ary \emph{projection} on $A$ is the operation $(a_1, \dots, a_n) \mapsto a_i$ for all $(a_1, \dots, a_n) \in A^n$. 

Let $f \colon A^n \to B$. For $i \in \nset{n}$, the $i$-th argument of $f$ is \emph{essential,} or $f$ \emph{depends} on the $i$-th argument, if there exist tuples $(a_1, \dots, a_n), (b_1, \dots, b_n) \in A^n$ such that $a_j = b_j$ for all $j \in \nset{n} \setminus \{i\}$ and $f(a_1, \dots, a_n) \neq f(b_1, \dots, b_n)$.

Two functions $f, g \colon A^n \to B$ are \emph{equivalent}, denoted $f \equiv g$, if there exists a bijection $\sigma \colon \nset{n} \to \nset{n}$ such that $f(a_1, \dots, a_n) = g(a_{\sigma(1)}, \dots, a_{\sigma(n)})$ for all $(a_1, \dots, a_n) \in A^n$.

Let $n \geq 2$, and let $f \colon A^n \to B$. For each $I \in \couples$, we define the function $f_I \colon A^{n-1} \to B$ by the rule
\[
f_I(a_1, \dots, a_{n-1}) =
f(a_1, \dots, a_{\max I - 1}, a_{\min I}, a_{\max I}, \dots, a_{n-1}).
\]
Note that $a_{\min I}$ occurs twice on the right side of the above equality, namely, at the two positions indexed by the elements of $I$. We will refer to the function $f_I$ as an \emph{identification minor} of $f$. This name is motivated by the fact that $f_I$ is obtained from $f$ by identifying the arguments indexed by the couple $I$.

\begin{lemma}[{Willard~\cite[Lemma~1.2]{Willard}}]
\label{lem:Willard1.2}
Let $A$ and $B$ nonempty sets, and let $f \colon A^n \to B$. Assume that $f$ depends on all of its arguments. If $n > \card{A}$, then there exists $I \in \couples$ such that $f_I$ depends on at least $n - 2$ arguments.
\end{lemma}

\subsection{Reconstruction problem for functions of several arguments}

Assume that $n \geq 2$ and let $f \colon A^n \to B$.
The \emph{deck} of $f$, denoted $\deck f$, is the multiset $\ms{f_I / {\equiv} : I \in \couples}$ of the equivalence classes of the identification minors of $f$.
Any element of the deck of $f$ is called a \emph{card} of $f$.
A function $g \colon A^n \to B$ is a \emph{reconstruction} of $f$, if $\deck f = \deck g$.
A function is \emph{reconstructible} if it is equivalent to all of its reconstructions.
A class $\cl{C} \subseteq \cl{F}_{AB}$ of functions is \emph{reconstructible} if all members of $\cl{C}$ are reconstructible.
A class $\cl{C} \subseteq \cl{F}_{AB}$ is \emph{weakly reconstructible} if for every $f \in \cl{C}$, all reconstructions of $f$ that are members of $\cl{C}$ are equivalent to $f$.
A class $\cl{C} \subseteq \cl{F}_{AB}$ is \emph{recognizable} if all reconstructions of the members of $\cl{C}$ are members of $\cl{C}$.
Note that if a class of functions is recognizable and weakly reconstructible, then it is reconstructible.

This reconstruction problem was formulated and some results, both positive and negative, on the reconstructibility of functions were presented in~\cite{LehtonenDeckSymm}. The reader is referred to this paper for more details, motivations and background information.

\subsection{On the reconstructibility of affine functions}

By a \emph{nonassociative right semiring} we mean an algebra $(G; +, \cdot)$ with binary operations $+$ and $\cdot$ called \emph{addition} and \emph{multiplication}, respectively, such that
\begin{itemize}
\item $(G; +)$ is a commutative monoid with neutral element $0$ ($0 + a = a + 0 = a$),
\item $(G; \cdot)$ is a groupoid with right identity $1$ ($a \cdot 1 = a$),
\item multiplication right distributes over addition ($(a + b) \cdot c = a \cdot c + b \cdot c$),
\item multiplication on the right by $0$ annihilates $G$ ($a \cdot 0 = 0$).
\end{itemize}
A nonassociative right semiring $(G; +, \cdot)$ is \emph{cancellative} if the additive monoid $(G; +)$ is cancellative, i.e., $a + b = a + c$ implies $b = c$.
We will denote multiplication simply be concatenation.

The attribute ``nonassociative'' refers to the fact that we do not require that multiplication be associative, contrary to the usual practice with semirings. The attribute ``right'' refers to the fact that we only stipulate right multiplicative identity, right distributivity, and right annihilation. A \emph{nonassociative left semiring} could be defined analogously, but we will not need this notion here.
Examples of nonassociative right semirings include semirings, rings, fields, and bounded distributive lattices. Rings and fields are cancellative.

A function $f \colon G^n \to G$ is \emph{affine} over $(G; +, \cdot)$ if
\begin{equation}
\label{eq:affine}
f(x_1, \dots, x_n) = a_1 x_1 + \dots + a_n x_n + c,
\end{equation}
for some $a_1, \dots, a_n, c \in G$. If $c = 0$, then $f$ is \emph{linear.}

\begin{lemma}
\label{lem:fgabcd}
Let $(G; +, \cdot)$ be a nonassociative right semiring.
Let $f$ be an affine function over $(G; +, \cdot)$.
If $f$ is linear or if $(G; +, \cdot)$ is cancellative, then $f$ has a unique representation of the form~\eqref{eq:affine}.
\end{lemma}

\begin{proof}
Let $f \colon G^n \to G$, $f(x_1, \dots, x_n) = a_1 x_1 + \dots + a_n x_n + c$. Assume that $f(x_1, \dots, x_n) = a'_1 x_1 + \dots + a'_n x_n + c'$ for some $a'_1, \dots, a'_n, c' \in G$. Then $c = f(0, \dots, 0) = c'$, and for every $i \in \nset{n}$,
$a_i + c = f(\mathbf{e}_i) = a'_i + c' = a'_i + c$, where $\mathbf{e}_i$ denotes the $n$-tuple in which the $i$-th entry is $1$ and the remaining entries are $0$.
If $f$ is linear (i.e., $c = 0$) or if $(G; +, \cdot)$ is cancellative, then $a_i = a'_i$ for all $i \in \nset{n}$.
\end{proof}

Let $f(x_1, \dots, x_n) = a_1 x_1 + \dots + a_n x_n + c$. Denote by $C_f$ the multiset $\ms{a_1, \dots, a_n}$ of the coefficients of the non-constant terms of $f$.

\begin{lemma}
\label{lem:fgequivCfCg}
Let $(G; +, \cdot)$ be a nonassociative right semiring.
Let $f, g \colon G^n \to G$ be affine functions over $(G; +, \cdot)$.
Assume that $f$ and $g$ are linear or $(G; +, \cdot)$ is cancellative.
Then $f \equiv g$ if and only if $C_f = C_g$ and the constant terms of $f$ and $g$ are equal.
\end{lemma}

\begin{proof}
Let $f(x_1, \dots, x_n) = \sum_{i = 1}^n a_i x_i + c$ and $g(x_1, \dots, x_n) = \sum_{i = 1}^n b_i x_i + d$.
Assume first that $f \equiv g$. Then there exists a permutation $\sigma \colon \nset{n} \to \nset{n}$ such that $f(a_1, \dots, a_n) = g(a_{\sigma(1)}, \dots, a_{\sigma(n)})$ for all $(a_1, \dots, a_n) \in G^n$. Thus, $f(x_1, \dots, x_n) = \sum_{i = 1}^n b_i x_{\sigma(i)} + d$. By Lemma~\ref{lem:fgabcd}, $c = d$ and $a_i = b_{\sigma^{-1}(i)}$ for all $i \in \nset{n}$. Thus, $C_f = \ms{a_1, \dots, a_n} = \ms{b_{\sigma^{-1}(1)}, \dots, b_{\sigma^{-1}(n)}} = C_g$.

For the converse implication, assume that $c = d$ and $C_f = C_g$. Then there exists a permutation $\sigma \colon \nset{n} \to \nset{n}$ such that $a_i = b_{\sigma(i)}$ for all $i \in \nset{n}$. We have
\begin{multline*}
f(x_1, \dots, x_n)
= a_1 x_1 + \dots + a_n x_n + c
= b_{\sigma(1)} x_1 + \dots + b_{\sigma(n)} x_n + d \\
= b_1 x_{\sigma^{-1}(1)} + \dots + b_n x_{\sigma^{-1}(n)} + d
= g(x_{\sigma^{-1}(1)}, \dots, x_{\sigma^{-1}(n)}).
\end{multline*}
Thus $f \equiv g$.
\end{proof}

For a multiset $M$ over $G$ and $c \in G$, with $\card{M} = n$, denote by $F_{M,c}$ the set $\{f \colon G^n \to G : C_f = M,\, f(0, \dots, 0) = c\}$.
It is clear from the definition and from Lemma~\ref{lem:fgequivCfCg} that $F_{M,c} = F_{M',c'}$ if and only if $M = M'$ and $c = c'$.

\begin{lemma}
\label{lem:functionsandmultisets}
Let $f \colon G^n \to G$ be an affine function over a nonassociative right semiring $(G; +, \cdot)$. Then
$\deck C_f = \ms{C_{f_I} : I \in \couples}$ and $\deck f = \ms{F_{M_I,c} : I \in \couples}$, where $c = f(0, \dots, 0)$ and $(M_I)_{I \in \couples}$ is an indexed family satisfying $\deck C_f = \ms{M_I : I \in \couples}$.
\end{lemma}

\begin{proof}
Let $f(x_1, \dots, x_n) = a_1 x_1 + \dots + a_n x_n + c$. Then $C_f = \ms{a_1, \dots, a_n}$ and $c = f(0, \dots 0)$.
For each $I \in \couples$, let $(C_f)_I := C_f \setminus \ms{a_{\min I}, a_{\max I}} \uplus \ms{a_{\min I} + a_{\max I}}$. Then $\deck C_f = \ms{(C_f)_I : I \in \couples}$. For each $I \in \couples$,
\begin{multline*}
f_I(x_1, \dots, x_{n-1}) = \\
(a_{\min I} + a_{\max I}) x_{\min I}
+ \sum_{i = 1}^{\min I - 1} a_i x_i
+ \sum_{i = \min I + 1}^{\max I - 1} a_i x_i
+ \sum_{i = \max I + 1}^n a_i  x_{i-1}.
\end{multline*}
Thus $C_{f_I} = (C_f)_I$.
We conclude that $\deck C_f = \ms{(C_f)_I : I \in \couples} = \ms{C_{f_I} : I \in \couples}$
and
$\deck f = \ms{f_I / {\equiv} : I \in \couples} = \ms{F_{C_{f_I},c} : I \in \couples} = \ms{F_{(C_f)_I,c} : I \in \couples}$.
\end{proof}

\begin{theorem}
\label{thm:affineweaklyrec}
Let $f, g \colon G^n \to G$ be affine functions over a nonassociative right semiring $(G; +, \cdot)$ with $n \geq 4$.
If $f$ and $g$ are linear or if $(G; +, \cdot)$ is cancellative, then $\deck f = \deck g$ if and only if $f \equiv g$.
\end{theorem}

\begin{proof}
Let
\begin{align*}
f(x_1, \dots, x_n) &= a_1 x_1 + \dots + a_n x_n + c, \\
g(x_1, \dots, x_n) &= b_1 x_1 + \dots + b_n x_n + d,
\end{align*}
for some $a_1, \dots, a_n, b_1, \dots, b_n, c, d \in G$. We assume that $c = d = 0$ or $(G; +, \cdot)$ is cancellative.

It is clear that if $f \equiv g$ then $\deck f = \deck g$.
Assume that $\deck f = \deck g$. Since $f_I(0, \dots, 0) = f(0, \dots, 0) = c$ and $g_I(0, \dots, 0) = g(0, \dots 0) = d$ for all $I \in \couples$, we must have that $c = d$.

By Lemma~\ref{lem:functionsandmultisets}, $\deck f = \ms{F_{M_I,c} : I \in \couples}$ and $\deck g = \ms{F_{M'_I,c} : I \in \couples}$, where $(M_I)_{I \in \couples}$ and $(M'_I)_{I \in \couples}$ are indexed families satisfying $\ms{M_I : I \in \couples} = \deck C_f$ and $\ms{M'_I : I \in \couples} = \deck C_g$.
Since $F_{M,c} = F_{M',c}$ if and only if $M = M'$, we have that $\deck C_f = \deck C_g$.
If $n \geq 5$, then Theorem~\ref{thm:ngeq5} implies that $C_f = C_g$.
Since $(G; +)$ is associative, Theorem~\ref{thm:neq4} implies, in light of Example~\ref{ex:4rstab}, that $C_f = C_g$ in the case that $n = 4$.
Applying Lemma~\ref{lem:fgequivCfCg}, we conclude that $f \equiv g$.
\end{proof}

Theorem~\ref{thm:affineweaklyrec} asserts that the class of linear functions of arity at least $4$ over any nonassociative right semiring $(G; +, \cdot)$ is weakly reconstructible. Furthermore, if $(G; +, \cdot)$ is cancellative, then the class of affine functions of arity at least $4$ over $(G; +, \cdot)$ is weakly reconstructible.

Let us consider the special case when $(G, +, \cdot)$ is a finite field of order $q = p^k$ ($p$ prime).
It is well known that every operation on a finite field is a polynomial function. Moreover, each function $f \colon G^n \to G$ is induced by a unique polynomial in $n$ variables where every exponent of every occurrence of every variable is at most $q - 1$. Such a polynomial is referred to as the \emph{canonical polynomial} of $f$.
It is easy to verify that a polynomial function $f$ depends on the $i$-th argument if and only if the variable $x_i$ occurs in the canonical polynomial of $f$.

\begin{lemma}
\label{lem:nonaffine}
Assume that $(G; +, \cdot)$ is a finite field of order $q = p^k$. If $n > \max(q, 3)$ and $f \colon G^n \to G$ is not affine, then there exists $I \in \couples$ such that $f_I$ is not affine.
\end{lemma}

\begin{proof}
The canonical polynomial of $f$ can be written as $P = \sum_{\mathbf{r} \in \{0, \dots, q - 1\}^n} a_\mathbf{r} \mathbf{x}^\mathbf{r}$, where $\mathbf{r} = (r_1, \dots, r_n)$ and $\mathbf{x}^\mathbf{r} = x_1^{r_1} x_2^{r_2} \cdots x_n^{r_n}$.
Let $P_\mathrm{aff}$ be the polynomial comprising the monomials of $P$ of total degree at most $1$, and let $P_\mathrm{non}$ be the polynomial comprising the monomials of $P_g$ of total degree at least $2$, i.e.,
\[
P_\mathrm{aff} = a_{\mathbf{0}} + \sum_{i = 1}^n a_{\mathbf{e}_i} x_i,
\qquad
P_\mathrm{non} = \sum_{\mathbf{r} \in \{0, \dots, q - 1\}^n \setminus \{\mathbf{0}, \mathbf{e}_1, \dots, \mathbf{e}_n\}} a_\mathbf{r} \mathbf{x}^\mathbf{r},
\]
where $\mathbf{0} = (0, \dots, 0)$ and $\mathbf{e}_i$ is the $n$-tuple in which the $i$-th entry is $1$ and the remaining entries are $0$. Let $f_\mathrm{aff}, f_\mathrm{non} \colon G^n \to G$ be the functions induced by the polynomials $P_\mathrm{agg}$ and $P_\mathrm{non}$, respectively. Then clearly $P = P_\mathrm{aff} + P_\mathrm{non}$ and $f = f_\mathrm{aff} + f_\mathrm{non}$ (pointwise addition of functions). Furthermore, for all $I \in \couples$, we have $f_I = (f_\mathrm{aff})_I + (f_\mathrm{non})_I$. Since $f$ is not affine, it holds that $P_\mathrm{non} \neq 0$.

Assume first that $P_\mathrm{non}$ has a monomial $M = a_\mathbf{r} \mathbf{x}^\mathbf{r}$ in which there occur at most $n - 2$ variables, i.e., $a_\mathbf{r} \neq 0$ and there exist $i, j \in \nset{n}$ such that $i \neq j$ and $r_i = r_j = 0$. Let $I = \{i, j\}$. The canonical polynomial of $f_I$ contains the monomial $M$ (with some reindexing of variables, if necessary); hence $f_I$ is not affine.

Assume then that all monomials in $P_\mathrm{non}$ have at least $n - 1$ variables. If there is a variable $x_i$ with $i \in \nset{n}$ that does not occur in any of the monomials of $P_\mathrm{non}$, then let $I = \{i, j\}$ for any $j \in \nset{n} \setminus \{i\}$. The canonical polynomial of $f_I$ contains all monomials of $P_\mathrm{non}$ (with some reindexing of variables, if necessary); hence $f_I$ is not affine.

We are left with the case that all monomials in $P_\mathrm{non}$ have at least $n - 1$ variables and all variables $x_i$, $i \in \nset{n}$, occur in $P_\mathrm{non}$.
Identification of a pair of variables in $P_\mathrm{non}$ results in a polynomial in which all monomials have at least $n - 2$ variables; some monomials may cancel each other, so the resulting polynomial may be $0$.
Since $f_\mathrm{non}$ depends on all of its $n$ arguments and $n > \max(q, 3)$, it follows from Lemma~\ref{lem:Willard1.2} that there exists $I \in \couples$ such that $(f_\mathrm{non})_I$ depends on at least $n - 2$ arguments; hence the canonical polynomial of $(f_\mathrm{non})_I$ cannot be $0$, so it contains a monomial with at least $n - 2$ variables.
Consequently, the canonical polynomial of $f_I$ has a monomial with at least $n - 2$ variables; hence $f_I$ is not affine.
\end{proof}

\begin{theorem}
\label{thm:affineoverfield}
Let $(G; +, \cdot)$ be a finite field of order $q = p^k$. The affine functions of arity at least $\max(q, 3) + 1$ over $(G; +, \cdot)$ are reconstructible.
\end{theorem}

\begin{proof}
Let $\cl{C}$ be the class of affine functions of arity at least $\max(q, 3) + 1$ over $(G; +, \cdot)$.
Since the identification minors of affine functions are affine, Lemma~\ref{lem:nonaffine} implies that $\cl{C}$ is recognizable. By Theorem~\ref{thm:affineweaklyrec}, $\cl{C}$ is weakly reconstructible. Consequently, $\cl{C}$ is reconstructible.
\end{proof}

\begin{remark}
The lower bound $\max(q, 3) + 1$ in Theorem~\ref{thm:affineoverfield} cannot be improved. As explained in~\cite{LehtonenDeckSymm}, no function $f \colon A^n \to B$ with $n \leq \card{A}$ is reconstructible.
It is also necessary to assume that the arity is greater than $3$. Since $(G; +)$ is a group, additive inverses exist for all elements, and for all $a, b \in G$, the multisets $\ms{a, b, -(a + b)}$ and $\ms{-a, -b, a + b}$ have the same deck (see Example~\ref{ex:3abaceqa}); thus the affine functions induced by the polynomials $a x_1 + b x_2 - (a + b) x_3$ and $-a x_1 - b x_2 + (a + b) x_3$ have the same deck.

Furthermore, if $(G; +, \cdot)$ is the two-element field, then $(G; +)$ is a Boolean group, and the multisets $\ms{1, 1, 1}$ and $\ms{1, 0, 0}$ have the same deck (see Example~\ref{ex:3abaceqbc}). Thus the ternary functions induced by the polynomials $x_1 + x_2 + x_3$ and $x_1$ have the same deck, because all identification minors of these functions are projections, and any two projections are equivalent. The class of affine functions of arity $3$ on the $2$-element field is not even recognizable. Namely, all identification minors of the function induced by the polynomial $x_1 x_2 + x_1 x_3 + x_2 x_3$ are projections, too.
\end{remark}

\begin{remark}
As explained in~\cite{LehtonenDeckSymm}, if $A$ is infinite, then no function $f \colon A^n \to B$ is reconstructible.
Even the class of polynomial functions over an infinite field $F$ fails to be weakly reconstructible. For $n \geq 2$, define the polynomial function $\Delta_n \colon F^n \to F$,
\[
\Delta_n(x_1, \dots, x_n) = \prod_{1 \leq i < j \leq n} (x_i - x_j).
\]
We have that $(\Delta_n)_I(x_1, \dots, x_{n-1}) = 0$ for every $I \in \couples$. Consequently, for any function $f \colon F^n \to F$ (polynomial or not), it holds that $f_I = (f + \Delta_n)_I$ for every $I \in \couples$ and $\deck f = \deck (f + \Delta_n)$.
\end{remark}

%%%%%%%%%%%%%%%%%%%%%%%%%%%%%%%%%%%%%%%%%%%%%%%%%%

\section*{Acknowledgments}

The author would like to thank Miguel Couceiro for inspiring discussions on minors of functions and reconstruction problems.

%%%%%%%%%%%%%%%%%%%%%%%%%%%%%%%%%%%%%%%%%%%%%%%%%%

\end{document}